\documentclass[10pt]{article}

\usepackage{amsmath}
\usepackage{amssymb}
\usepackage{moreverb}

\usepackage[all]{xy} 

\catcode`\à= \active \def à{\`a}
\catcode`\À= \active \def À{\`A}
\catcode`\é= \active \def é{\'e}
\catcode`\É= \active \def É{\'E}
\catcode`\è= \active \def è{\`e}
\catcode`\È= \active \def È{\`E}
\catcode`\ù= \active \def ù{\`u}
\catcode`\â= \active \def â{\^a}
\catcode`\Â= \active \def Â{\^A}
\catcode`\ê= \active \def ê{\^e}
\catcode`\Ê= \active \def Ê{\^E}
\catcode`\î= \active \def î{\^\i}
\catcode`\ô= \active \def ô{\^o}
\catcode`\û= \active \def û{\^u}
\catcode`\ä= \active \def ä{\"a}
\catcode`\ë= \active \def ë{\"e}
\catcode`\ï= \active \def ï{\"\i}
\catcode`\ö= \active \def ö{\"o}
\catcode`\ü= \active \def ü{\"u}
\catcode`\ç= \active \def ç{\c{c}}
\catcode`\Ç= \active \def Ç{\c{C}}

\newcounter{remark}
\newcommand{\remark}{\addtocounter{remark}{1}
                       \par \quad {\bf \arabic{remark}}.\,
                      }

\newenvironment{rk}{\begin{quote}
                     \normalfont\footnotesize {{\bf Remark} --}
                    }{\end{quote}}

\newenvironment{proof}{\medbreak \noindent {\bf Proof~---}}
                       {\hfill $\square$ \medbreak}

\newcommand{\Cl}{\operatorname{\mathcal{C}\ell}}

\newcommand{\dis}{\operatorname{dis}}
\newcommand{\Div}{\operatorname{Div}}

\newcommand{\Gal}{\operatorname{Gal}}

\newcommand{\Id}{\operatorname{Id}}

\newcommand{\Ker}{\operatorname{Ker}}

\newcommand{\id}{\mathfrak}
\newcommand{\ideng}[1]{\left\langle #1 \right\rangle}

\newcommand{\NN}{\mathbb N}
\newcommand{\ZZ}{\mathbb Z}

\newtheorem{theo}{Theorem}

\newtheorem{cor}[theo]{Corollary}
\newtheorem{lem}[theo]{Lemma}

\newtheorem{prop}[theo]{Proposition}

\title{\bfseries On the kernel of the norm in some unramified number fields extensions}

\author{
Emmanuel {\sc Hallouin}
and Marc {\sc Perret}\thanks{Laboratoire Emile Piacrd, Institut de Mathématiques de Toulouse, France.}
}

\begin{document}
\maketitle

\section*{Introduction}

Let~$L/K$ be a unramified Galois extension of number fields whose Galois group~$G$ is a finite
$p$-group ($p$ a prime integer). In~\cite{Serre_Galois_Cohomology}, Chap~I, \S 4.4, it is
proved that if~$L$ is principal then:
\begin{align} \label{r-d}
d_p H^3(G,  \ZZ) = d_p H^2(G,\ZZ/p\ZZ) - d_p H^1(G, \ZZ/p\ZZ) \leq r_1 + r_2
\end{align}
where~$d_p G$ denotes the $p$-rank of a finite $p$-group~$G$ and where~$(r_1, r_2)$ is the
signature of the number field~$K$. Briefly, the proof works as follows.
Let~$C_L$ be the id\`eles class group of~$L$ and~$E_L$ its unit group, then:
$$
\forall q \in \ZZ,
\qquad
H^q(G, C_L) \simeq H^{q+1}(G, E_L)
\quad\text{and}\quad
H^q(G, C_L) \simeq H^{q-2}(G, \ZZ).
$$
The first isomorphism follows from the fact that~$L$ is principal while the second one
is part of class field theory. Thus:
\begin{equation} \label{isom}
H^{q+1}(G, E_L) \simeq H^{q-2}(G, \ZZ).
\end{equation}
The inequality~(\ref{r-d}) comes from the specialization at~$q=-1$ of this isomorphism because
the rank of~$H^0(G, E_L)$ is easily bounded thanks to Dirichlet's units theorem.

Together with Golod-Safarevich's group theoretic result,~(\ref{r-d})  implies that if a
number field~$K$ satisfies the quadratic (in~$d_p \Cl(K)$) inequality:
$$
d_p \Cl(K)^2 - d_p \Cl(K) > r_1 + r_2 - 1
$$
then its $p$-class field tower is infinite.

In order to find a cubic (in~$d_p \Cl(K)$) analogue of this criteria,
we specializes the isomorphism~(\ref{isom}) at~$q = -2$. This yields the following equality:
$$
d_p H^{-1}(G, E_L) = d_p H^3(G,\ZZ/p\ZZ) - d_p H^2(G,\ZZ/p\ZZ) + d_p H^1(G,\ZZ/p\ZZ).
$$
It is so crucial to find an upperbound for the $p$-rank~$d_p H^{-1}(G, E_L)$
when~$\Cl(L)$ is trivial. In this paper, we prove results about this rank in some
special cases. More precisely, we compute this $p$-rank when~$L/K$ is an abelian unramified
(also at infinity) $p$-extension
whose Galois group can be generated by two elements. We also exhibit an explicit basis of
the $p$-group~$H^{-1}(G, E_L)$.

\medbreak

{\bfseries Notations ---}
Let~$K$ be a number field. We denote by~$\Sigma_K$ the set finite
places,~$\Div(K)$ its divisor group and~$\Cl(K)$ its divisor class group.
To each finite place~$v \in \Sigma_K$ one can associate a unique
prime ideal~$\id{p}_v$ of~$K$ and to each~$x \in K$, there corresponds
a principal divisor~$\ideng{x}_K$ of~$K$. 

If~$L/K$ is a Galois extension of number fields, then for
each~$v \in \Sigma_K$,~$\Sigma_{L,v}$ denotes the subset
of places~$w \in \Sigma_L$ above~$v$ (for short~$w \mid v$) and~$f_v$
the residual degree of any~$w \in \Sigma_{L,v}$ over~$K$. The
map~$e_{L/K} : \Div(K) \to \Div(L)$ is the classical extension of ideals.

Let~$G$ be a finite group and~$M$ be a $G$-module. The norm map~$N_G : M \to M$
is defined by~$x \mapsto \prod_{g \in G} g(x)$; its kernel is denoted
by~$M[N_G]$. The augmentation
ideal~$I_G M = \ideng{\frac{g(x)}{x}, x \in M, g \in G}$
is of importance. Of course, one has~$I_G M \subset M[N_G]$; the quotient
of these two subgroups is nothing else that the Tate cohomology group:
$$
H^{-1}(G, M) \overset{\text{def.}}{=} \frac{M[N_G]}{I_G M}
$$
in which we are interested (see~\cite{Serre_CorpsLocaux} for an introduction
to the negative cohomology groups).

\section{The cyclic case} \label{s_cyclic}

Let~$L/K$ be a cyclic extension with Galois group~$G = \ideng{g}$.  A classical consequence of
the Hilbert's~90 theorem states that the kernel of the norm~$N_G$ equals the augmentation
ideal:~$L^*[N_G] = I_G L^*$. In cohomological terms, this means that:
$$
H^1(G, L^*) = \{1\}
\qquad \Longrightarrow \qquad
H^{-1}(G, L^*) = \{1\}.
$$
Another easy consequence already known is that:

\begin{prop}
Let~$L/K$ be a cyclic unramified extension with Galois group~$G = \ideng{g}$. Then the map:
$$
\begin{array}{rccc}
\varphi_g : & \Ker(\Cl(K) \to \Cl(L)) & \longrightarrow & H^{-1}(G, E_L) \\
            &     [I]                & \longmapsto      & \frac{g(y)}{y}
\end{array},
$$
where~$[I]$ denotes the ideal class of~$I$ and~$y$ any generator of~$I$ in~$L$, is an
isomorphism of groups.
\end{prop}

\begin{proof}
The only non-trivial assertion to verify is the surjectivity of the map. Let~$u \in E_N[N_G]$, then
there exists~$y \in L^*$ such that~$u = \frac{g(y)}{y}$. Thus
the ideal~$\ideng{y}_L$ is fixed by the action of~$G$. The extension~$L/K$ being unramified, the
ideal~$\ideng{y}_L$ is the extension to~$L$ of an ideal~$I$
of~$K$:~$e_{L/K}(I) = \ideng{y}_L$.
Then~$u = \varphi_g([I])$.
\end{proof}

This proposition implies the following corollary:

\begin{cor} \label{H-1_unites_cyclic}
Let~$K$ be a number field and~$L/K$ an unramified (included at infinity) abelian extension
with Galois group~$G$ a cyclic $p$-group such that~$L$ is principal.
If~$G = \ideng{g}$ and if~$\pi$ generate a prime ideal
of~$L$ with Frobenius equal to~$g$, then:
$$
H^{-1}(G, E_L)
=
\ideng{\frac{g(\pi)}{\pi}}.
$$
\end{cor}

\section{Some experiments with {\tt magma}}

With the help of {\tt magma} and {\tt pari/gp}, we have made some experiments and collect
informations about
the $2$-rank of the group~$H^{-1}(G, E_{K^i})$ in unramified finite $2$-extensions~$K^i/K$
($i = 1,2$).
In each case, we start with a quadratic complex number field~$K$ whose class group is a $2$-group;
tables of such fields can be found in~\cite{Lemmermeyer}.
We compute~$K^1 = K^{\rm hilb.}$ and the group structure
of~$H^{-1}(E_{K^1}) \overset{\rm def.}{=} H^{-1}(\Gal(K^1/K), E_{K^1})$. 
If~$\Cl(K^1)$ is not trivial, we try to go further. We compute~$K^2 = (K^1)^{hilb.}$
and the group structure of~$H^{-1}(E_{K^2}) \overset{\rm def.}{=} H^{-1}(\Gal(K^2/K), E_{K^2})$.

Here is our {\tt magma} program we used:
\begin{center}
\begin{minipage}{16cm}
{\small
\begin{boxedverbatim}
clear ;
Q := RationalField() ;
dis := -84 ;
K<x> := QuadraticField(dis) ;

"Computation of K^hilb..." ;
Khilb := AbsoluteField(HilbertClassField(K)) ;
Khilb<y> := OptimizedRepresentation(Khilb) ;

"... compuation of the unit group of K^hilb..." ;
E_Khilb, e_Khilb := UnitGroup(Khilb) ;

Gal_Khilb_Q, Aut_Khilb_Q, i := AutomorphismGroup(Khilb) ;
G := FixedGroup(Khilb, K) ;
Norm_G := map < Khilb -> Khilb | y :-> &* [i(g)(y) : g in G] > ;
N := hom < E_Khilb -> E_Khilb | [(e_Khilb * Norm_G * Inverse(e_Khilb))(E_Khilb.i) :
                                          i in [1..NumberOfGenerators(E_Khilb)]] > ;

Ker_N := Kernel(N) ;
I_G := [i(g)(u)/u : u in Generators(E_Khilb) @ e_Khilb, g in G] ;
I_G := sub < E_Khilb | I_G @@ e_Khilb > ;
assert(I_G subset Ker_N) ;
printf "... structure of H^(-1)(G, E_M) = %o\n", Ker_N / I_G ;
\end{boxedverbatim}
}
\end{minipage}
\end{center}

Unfortunately, because of the difficulty of computing the unit group of a number field, 
only few computations achieved. In the following table, the notation~$2 \cdot 4$ means that
the concerning group is isomorphic to~$\ZZ/2\ZZ \times \ZZ/4\ZZ$.
$$
\begin{array}{c|c||c|c||c|c}
\dis(K) & \Cl(K) & \Cl(K^1) & H^{-1}(E_{K^1}) & \Cl(K^2) & H^{-1}(E_{K^2}) \\
\hline
-84 & 2 \cdot 2 & 1 & 2 \cdot 2 \cdot 2 & & \\
\hline
-120 & 2 \cdot 2 & 2 & 4 & 1 & 8 \\
\hline
-260 & 2 \cdot 4 & 2 & 2 \cdot 4 & 1 & 2 \cdot 8 \\
\hline
-280 & 2 \cdot 2 & 4 & 4         & 1 & 16 \\ 
\hline
-308 & 2 \cdot 4 & 1 & 2 \cdot 2 \cdot 4 & &  \\
\hline
-399 & 2 \cdot 8 & 1 & 2 \cdot 2 \cdot 8 & & \\
\hline
-408 & 2 \cdot 2 & 2 & 2 \cdot 2 \cdot 2 & 1& 2 \cdot 2 \cdot 4 \\
\hline
-420 & 2 \cdot 2 \cdot 2 & 2 \cdot 2 & 2 \cdot 2 \cdot 2 \cdot 4 & 1 & \text{unkown} \\
\end{array}
$$

In the following section, we will explain why~$d_2 H^{-1}(E_{K^1}) = 3$
when~$d_2 \Cl(K) = 2$ and~$d_2 \Cl(K^1) = 1$. In all the remaining known cases, we point out
that~$d_2 H^{-1}(E_{K^1}) = d_2 H^{-1}(E_{K^2})$.

\section{When the Galois group has two generators}

The goal of is section is to extend the results of \S\ref{s_cyclic} to the case of extensions
whose Galois group is an abelian group generated by two elements. 

First, we investigate the cohomology group with values in~$M^*$. We still have: 

\begin{theo} \label{H-1_M_d_p_egal_2}
Let~$K$ be a number field and~$M/K$ be an unramified (included at infinity) extension
whose Galois group~$G$ is an abelian $p$-group generated by two elements.
Then~$H^{-1}(G, M^*) = 1$.
\end{theo}

\begin{proof}
Since~$M/K$ is an abelian unramified extension, there exists~$G'$ a subgroup of~$\Cl(K)$ such
that~$G \simeq \Cl(K) / G'$. Let~$\id{p}_1, \ldots, \id{p}_r$ be primes of~$K$ whose classes
generate~$G'$.
 If~$G \simeq \ZZ/p^\alpha\ZZ \times \ZZ/p^\beta\ZZ$ with~$\alpha \leq \beta$, we complete these
primes
by choosing~$\id{p}, \id{q}$ primes of~$K$ such
that their decomposition groups in~$M/K$
satisfy~$D(\id{p}) = \ideng{(1,1)}$ and~$D(\id{q}) = \ideng{(0,1)}$.
Adjoining~$\id{p}, \id{q}$ to the~$\id{p}_i$'s leads to a system of generators of~$\Cl(K)$.

Let~$H = \ideng{(1,0)}$. Then~$H$ and~$G/H$ are cyclic and, by construction, the decomposition
groups in~$M/K$ satisfy:
$$
\forall 1 \leq i \leq r, \quad D(\id{p}_i) \cap H = \{\text{id}\},
\qquad
D(\id{p}) \cap H = \{\text{id}\},
\qquad
D(\id{q}) \cap H = \{\text{id}\}.
$$
Theorem~\ref{H-1_M_d_p_egal_2} is implied by the two following lemmas.
\end{proof}

\begin{lem}
Let~$H$ be a normal cyclic subgroup of~$G$. Then:
$$
H^{-1}(G, M^*) = \{1\}
\Longleftrightarrow
H^{-1}(G/H, N_H(M^*)) = \{1\}.
$$
\end{lem}

\begin{proof}
Suppose that~$H^{-1}(G, M^*) = \{1\}$. If~$y \in N_H(M^*)[N_{G/H}]$, then there
exists~$z \in M^*$ such that~$y = N_H(z)$ and~$N_G(z) = N_{G/H}(N_H(z)) = N_{G/H}(y) = 1$.
Thus, by hypothesis,~$z \in M^*[N_G] = I_G M^*$:
$$
\exists z_i \in M, \; g_i \in G, \quad
z = \frac{g_1(z_1)}{z_1} \times\cdots\times \frac{g_r(z_r)}{z_r}.
$$
Hence:
$$
y = N_H(z) = \frac{g_1(N_H(z_1))}{N_H(z_1)} \times\cdots\times \frac{g_r(N_H(z_r))}{N_H(z_r)}.
$$
Therefore~$y \in I_{G/H} N_H(M^*)$.

Conversely, suppose that~$H^{-1}(G/H, N_H(M^*)) = \{1\}$. If~$z \in M^*[N_G]$
then~$1 = N_G(z) = N_{G/H}(N_H(z))$ and thus~$N_H(z) \in N_H(M^*)[N_{G/H}]$.
By hypothesis, there
exist~$z_1, \ldots, z_r \in M^*$ and~$g_1, \ldots g_r \in G$ such that:
$$
N_H(z)
=
\frac{g_1(N_H(z_1))}{N_H(z_1)} \times \cdots \times \frac{g_r(N_H(z_r))}{N_H(z_r)}
=
N_H\left(\frac{g_1(z_1)}{z_1} \times \cdots \times \frac{g_r(z_r)}{z_r}\right).
$$
It follows that:
$$
z \in I_G M^* \times M^*[N_H] = I_G M^* \times I_H M^* = I_G M^*,
$$
because,~$H$ being cyclic, one has~$M^*[N_H] = I_H M^*$.
\end{proof}

\begin{lem}
Let~$H$ be a cyclic subgroup of~$G$ such that~$G/H$ is also cyclic.
If~$\Cl(K)$ can be generated by primes whose decomposition groups intersect~$H$ trivially,
then~$H^{-1}(G/H, N_H(M^*)) = \{1\}$.
\end{lem}

\begin{proof}
Let~$h$ be a generator of~$H$ and~$g \in G$ such that~$G = \ideng{g,h}$. Let~$L = M^H$ so
that~$\Gal(L/K) = \ideng{g}$.

Let~$y \in N_H(M^*)[N_{G/H}]$. Since~$G/H$ is cyclic generated by~$g$, there exists~$b \in L$
such that~$y = \frac{g(b)}{b}$.

Since~$y \in N_H(M^*)$, it is a norm everywhere locally:
\begin{align*}
\forall w \in \Sigma_L, \; w(y) \equiv 0 \pmod{f_w}
&\quad \Longrightarrow \quad
\forall w \in \Sigma_L, \; w \circ g(b) \equiv w(b) \pmod{f_w} \\
&\quad \Longrightarrow \quad
\forall v \in \Sigma_K, \; \forall w,w' \in \Sigma_{L,v}, \;
w'(b) \equiv w(b) \pmod{f_w}.
\end{align*}
Note that there is no condition at infinity because infinite places are supposed unramified.
The last assertion implies that the ideal~$J$ of~$L$ defined by:
$$
J = \prod_{w \in \Sigma_L} \id{p}_w^{-w(b) \bmod{f_w}}
\qquad \text{(for~$x \in \ZZ$, we choose~$x \bmod f_w \in [0..f_w - 1]$)},
$$
is the extension to~$L$ of the ideal~$I$ of~$K$ defined by:
$$
I = \prod_{v \in \Sigma_K} \id{p}_v^{-w(b) \bmod{f_w}}
\qquad \text{(for each~$v \in \Sigma_K$, we choose~$w$ a place of~$\Sigma_{L,v}$).}
$$

By hypothesis,~$\Cl(K)$ can be generated by prime ideals~$\id{p}_1, \ldots, \id{p}_r$ of~$K$
whose decomposition groups satisfy~$D(\id{p}_i) \cap H = \{\text{id}\}$. This means that all the
primes of~$L$ above the~$\id{p}_i$ totally split in~$M$. There exists~$a \in K$
and~$e_1, \ldots, e_r \in \NN$ such that~$\ideng{a} = I \times \prod_i \id{p}_i^{e_i}$.
By construction, the ideal~$ab$ of~$L$ has support on primes of~$L$ totally split in~$M$.

Recall that, in a cyclic extension, the local-global principle is true form norm equations.
Thus, by this local-global principle, we deduce that~$ab \in N_H(M^*)$. Finally,
because~$a \in K$, we have:
$$
y = \frac{g(b)}{b} = \frac{g(ab)}{ab} \in I_{G/H} N_H(M^*),
$$
which was to be proved.
\end{proof}

Secondly, as in the cyclic case, one can ask if the triviality of the cohomological group with values in~$M^*$ could imply some results about the cohomological group
with values in~$E_M$.

\begin{prop} \label{d_p-H-1_E_M_M_principal}
Let~$K$ be a number field and~$M/K$ an unramified (included at infinity) abelian extension
with Galois group~$G$ a $p$-group of $p$-rank~$d$. If~$M$ is principal,
then~$d_p H^{-1}(G, E_M) = \frac{d(d^2+5)}{6}$.
\end{prop}

\begin{proof}
In~\cite{Serre_Galois_Cohomology} \S4.4, using class field theory, it is proved that:
$$
\forall q \in \ZZ, \; H^{q+1}(G, E_M) \simeq H^{q-2}(G, \ZZ).
$$
Hence, for~$q=-2$, we obtain:
$$
H^{-1}(G, E_M) \simeq H^{-4}(G, \ZZ).
$$
By duality, it is enough to compute the $p$-rank of~$H^4(G, \ZZ)$. This can be done, starting
with the exact sequence of~$G$-modules
(trivial action)~$0 \to \ZZ \overset{p}{\to} \ZZ \to \ZZ/p\ZZ \to 0$ and considering the long
cohomology exact sequence:
\begin{align*}
0 \to H^1(G, \ZZ/p\ZZ) \to &H^2(G, \ZZ/p\ZZ) \overset{p}{\to} H^2(G, \ZZ/p\ZZ)
\to H^2(G, \ZZ/p\ZZ) \to \\
&H^3(G, \ZZ/p\ZZ) \overset{p}{\to} H^3(G, \ZZ/p\ZZ)
\to H^3(G, \ZZ/p\ZZ) \to H^4(G, \ZZ)[p] \to 0.
\end{align*}
The logarithm of the product of the orders of these groups equals~$0$, therefore:
$$
d_p H^4(G, \ZZ)
=
d_p H^3(G, \ZZ/p\ZZ) - d_p H^2(G, \ZZ/p\ZZ) + d_p H^1(G, \ZZ/p\ZZ)
$$
(recall that in a finite abelian $p$-group~$A$, one has:~$\# A[p] = p^{d_p A}$).
It is now easy to conclude because:
$$
d_p H^2(G, \ZZ/p\ZZ) = \frac{d(d+1)}{2}
\qquad\text{and}\qquad
d_p H^3(G, \ZZ/p\ZZ) = \frac{d(d+1)(d+2)}{6}
$$
as it can be proved using Künneth's formula
(see~\cite{Neukirch_Schmidt_Wingberg}, exercice~7, page~96).
\end{proof}

\begin{rk}
The isomorphism of the beginning of this proof specialized to~$q = -1$ is a key step
of the proof of the Golod-Shafarevich's theorem.
\end{rk}

Let us return to the case where~$d_p(G) = 2$. Then, due to
proposition~\ref{d_p-H-1_E_M_M_principal}, one has~$d_p(G, E_M) = 3$.
As in corollary~\ref{H-1_unites_cyclic}, one can be more precise and
exhibit a basis of~$H^{-1}(G, E_M)$.

\begin{theo}
Let~$K$ be a number field and~$M/K$ an unramified (included at infinity) abelian extension
with Galois group~$G$ a $p$-group of rank~$2$ such that~$M$ is principal.
If~$G = \ideng{g_1, g_2}$ and if~$\pi_1, \pi_2, \pi_{12}$ generate primes ideals
of~$M$ with Frobenius equal to~$g_1, g_2$ and~$g_1g_2$ respectively, then:
$$
H^{-1}(G, E_M)
=
\ideng{\frac{g_1(\pi_1)}{\pi_1},\frac{g_2(\pi_2)}{\pi_2},\frac{g_1g_2(\pi_{12})}{\pi_{12}}}.
$$
\end{theo}

\begin{proof}
{\bfseries First step.} We claim that~$H^{-1}(G, E_M)$ is generated by:
$$
H^{-1}(G, E_M)
=
\ideng{\frac{\sigma_{\pi}(\pi)}{\pi}, \; \text{$\pi$ a prime element of~$M$}}.
$$
where~$\sigma_\pi$ denotes the Frobenius at~$\pi$.

Let~$\pi$ be a prime element of~$M$ and~$g, g' \in G$ such
that~$g \equiv g' \bmod{D(\pi)}$ where~$D(\pi)$ denotes the decomposition group of the
ideal~$\ideng{\pi}_M$. Then there exists~$\alpha \in \NN$ such
that~$g^{-1}g' = \sigma_\pi^\alpha$ and thus:
$$
\frac{g'(\pi)}{g(\pi)}
=
g\left(\frac{g^{-1}g'(\pi)}{\pi}\right)
=
g\left(\frac{\sigma_\pi^\alpha(\pi)}{\pi}\right)
\equiv
\frac{\sigma_\pi^\alpha(\pi)}{\pi}
\equiv
\left(\frac{\sigma_\pi(\pi)}{\pi}\right)^\alpha \pmod{I_G E_M}.
$$

For every~$v \in \Sigma_K$, we choose a generator~$\pi_v$ of one of the primes of~$M$
above~$\id{p}_v$ and we fix a section~$\sigma \mapsto \widetilde{\sigma}$ of the cononical
projection map~$G \to G/D(\pi_v)$. The elements~$\widetilde{\sigma}(\pi_v)$,
when~$v$ runs in~$\Sigma_K$ and~$\sigma \in G/D(v)$, describe a system of prime elements of~$M$.
Then every~$z \in M$ factorizes into: 
$$
z = u \prod_{v \in \Sigma_K} \left(\prod_{\sigma \in G/D(v)} \widetilde{\sigma}(\pi_v)^{e_{v,\sigma}}\right)
\qquad\Longrightarrow\qquad
g(z) = g(u) \prod_{v \in \Sigma_K} \left(\prod_{\sigma \in G/D(v)} g\widetilde{\sigma}(\pi_v)^{e_{v,\sigma}}\right)
$$
for every~$g \in G$. Of course~$g\widetilde{\sigma} \equiv \widetilde{g\sigma} \bmod{D(\pi_v)}$
therefore there exists~$\alpha_{v,\sigma} \in \NN$ such that:
\begin{align*}
g\widetilde{\sigma}(\pi_v) &= \left(\frac{\sigma_v(\pi_v)}{\pi_v}\right)^{\alpha_{v,\sigma}}
\widetilde{g\sigma}(\pi_v)\\
&\Longrightarrow\qquad
g(z) \in \ideng{g(u)}\ideng{\frac{\sigma_{\pi}(\pi)}{\pi}, \; \text{$\pi$ a prime element of~$M$}}\ideng{\widetilde{\sigma}(\pi_v), \, v \in \Sigma_K, \sigma \in G/D(v)}.
\end{align*}

Now start with~$u \in E_M[N_G]$. By theorem~\ref{H-1_M_d_p_egal_2}, we know
that~$H^{-1}(G, M^*) = \{1\}$, i.e.~$M^*[N_G] = I_G M^*$.
Hence, there exists~$z_1, z_2 \in M^*$ such
that~$u = \frac{\sigma_1(z_1)}{z_1}\frac{\sigma_1(z_2)}{z_2}$.
Factorizing~$z_1$ and~$z_2$ into primes of~$M$ of the form~$\widetilde{\sigma}(\pi_v)$, one shows
that:
$$
u \in I_G E_M \ideng{\frac{\sigma_{\pi}(\pi)}{\pi}, \; \text{$\pi$ a prime element of~$M$}}\ideng{\widetilde{\sigma}(\pi_v), \, v \in \Sigma_K, \sigma \in G/D(v)};
$$
But, in this decomposition, since~$u$ is invertible, the element in the third group must be
equal to~$1$.

\medbreak

{\bfseries Second step.} We consider a prime element~$\pi$ of~$M$ whose Frobenius is denoted
by~$\sigma_\pi$. Let us prove that the class modulo~$I_G E_M$ of the
element~$u = \frac{\sigma_\pi(\pi)}{\pi}$ is contained in the subgroup generated by
the~$\frac{g_i(\pi_i)}{\pi_i}$ for~$i = 1,2,12$.

To this end, put~$H = \ideng{g_{12}}$,~$L = M^H$ and~$\id{p} = \ideng{\pi}_M \cap K$,
$\id{p}_1 = \ideng{\pi_1}_M \cap K$, $\id{p}_2 = \ideng{\pi_2}_M \cap K$.

There exits~$\alpha_1, \alpha_2 \in \NN$
such that~$\sigma_\pi = g_1^{\alpha_1}g_2^{\alpha_2}$ and,
by Artin map,~$\id{p} = a \id{p}_1^{\alpha_1}\id{p}_2^{\alpha_2}$ with~$a \in K^*$.
Since~$\ideng{\sigma_i} \cap H = \{\Id\}$ for~$i = 1,2$, the primes~$\id{p}_i$,
$i=1,2$, totally split between~$L$ and~$M$. Thus:
$$
\begin{cases}
e_{L/K}(\id{p}) = \ideng{N_H(\pi)}_L \\
e_{L/K}(\id{p_i}) = \ideng{N_H(\pi_i)}_L, \; i=1,2
\end{cases}
\qquad\Longrightarrow\qquad
N_H(\pi) = a v N_H(\pi_1)^{\alpha_1}N_H(\pi_2)^{\alpha_2},
$$
where~$v \in E_L$. Hence:
$$
N_H(u)
=
N_H\left(\frac{\sigma_\pi(\pi)}{\pi}\right)
=
\frac{\sigma_\pi\left(N_H(\pi)\right)}{N_H(\pi)}
=
\frac{\sigma_\pi(a)}{a}\frac{\sigma_\pi(v)}{v}
N_H\left(\frac{\sigma_\pi(\pi_1)}{\pi_1}\right)^{\alpha_1}
N_H\left(\frac{\sigma_\pi(\pi_2)}{\pi_2}\right)^{\alpha_2}.
$$
Let us look separately, at the four terms in the right hand product. The first one is equal to~$1$ because~$a \in K$.
Since local-global principal occurs in cyclic extensions and since~$M/L$ is unramified,
there exists~$w \in E_M$ such that~$v = N_H(w)$. Thus the second
term~$\frac{\sigma_\pi(v)}{v}$ equals~$N_H\left(\frac{\sigma_\pi(w)}{w}\right)$.
The thirst and fourth terms go in the same way:
since~$g_1, g_2$ generate~$G$, the elements~$g_1$ and~$g_1g_2$ also generate~$G$ and there
exists~$\beta_1, \beta_2 \in \NN$ such that~$\sigma_\pi = g_1^{\beta_1}(g_1g_2)^{\beta_2}$.
It follow that:
$$
N_H\left(\frac{\sigma_\pi(\pi_1)}{\pi_1}\right)
=
N_H\left(\frac{g_1^{\beta_1}(\pi_1)}{\pi_1}\right)
=
N_H\left(\frac{g_1(w_1)}{w_1}\left(\frac{g_1(\pi_1)}{\pi_1}\right)^{\beta_1}\right)
$$
where~$w_1 \in E_M$.

In conclusion, going back to~$u$, it satisfies:
\begin{align*}
N_H(u)
&=
N_H\left(
\frac{\sigma_\pi(w)}{w}
\frac{g_1(w_1)}{w_1}^{\alpha_1}
\frac{g_2(w_2)}{w_2}^{\alpha_1}
\left(\frac{g_1(\pi_1)}{\pi_1}\right)^{\alpha_1\beta_1}
\left(\frac{g_2(\pi_2)}{\pi_2}\right)^{\alpha_2\beta_2}
\right) \\
&\Longrightarrow
u \times
\left(
\frac{\sigma_\pi(w)}{w}
\frac{g_1(w_1)}{w_1}^{\alpha_1}
\frac{g_2(w_2)}{w_2}^{\alpha_2}
\left(\frac{g_1(\pi_1)}{\pi_1}\right)^{\alpha_1\beta_1}
\left(\frac{g_2(\pi_2)}{\pi_2}\right)^{\alpha_2\beta_2}
\right)^{-1} \in E_M^*[N_H].
\end{align*}
Finally, due to the cyclic case, we know
that~$E_M^*[N_H] = I_H E_M \ideng{\frac{g_1g_2(\pi_{12})}{\pi_{12}}}$
and thus:
$$
u \bmod{I_G E_M} \in
\ideng{\frac{g_1(\pi_1)}{\pi_1},\frac{g_2(\pi_2)}{\pi_2},\frac{g_1g_2(\pi_{12})}{\pi_{12}}},
$$
which was to be proved.
\end{proof}

\begin{rk}
All these results hold in the function field case for $S$-units where~$S$ is any non-empty
finite set of places.
\end{rk}


\end{document}